\documentclass[a4paper,10pt]{article}

\usepackage{amsmath, amsthm, amscd, amsfonts, amssymb, graphicx, color}
\usepackage[bookmarksnumbered, colorlinks, plainpages]{hyperref}

\usepackage{enumitem}
\usepackage{url}
\usepackage{multirow}
\usepackage{subfigure}
\usepackage[pass]{geometry}
\usepackage{cite}
\usepackage{cancel}

\usepackage{tikz}
\usetikzlibrary{calc,shapes,decorations.pathreplacing,matrix}

\tikzset{square matrix/.style={
    matrix of nodes,
    column sep=-\pgflinewidth, row sep=-\pgflinewidth,
    nodes={draw,
      minimum height=4.5pt,
      anchor=center,
      text width=4.5pt,
      align=center,
      inner sep=0pt
    },
  },
  square matrix/.default=1.2cm
}

\newtheorem{thm}{Theorem}

\newtheorem{obs}{Observation}
\newtheorem{defn}{Definition}
\newtheorem{cor}{Corollary}
\newtheorem{conj}{Conjecture}
\newtheorem{prop}{Proposition}
\newtheorem{ques}{Question}

\begin{document}

\title{Upper $k$-tuple total domination in graphs}

\author{Adel P. Kazemi\\[1em]
Department of Mathematics\\ University of Mohaghegh Ardabili \\ P.O.\ Box 5619911367, Ardabil, Iran. \\
Email: adelpkazemi@yahoo.com \\[1em]
}

\maketitle

\begin{abstract}
Let $G=(V,E)$ be a simple graph. For any integer $k\geq 1$, a subset
of $V$ is called a $k$-tuple total dominating set of $G$ if every
vertex in $V$ has at least $k$ neighbors in the set. The minimum
cardinality of a minimal $k$-tuple total dominating set of $G$ is
called the $k$-tuple total domination number of $G$. In this paper,
we introduce the concept of upper $k$-tuple total domination number
of $G$ as the maximum cardinality of a minimal $k$-tuple total
dominating set of $G$, and study the problem of finding a minimal
$k$-tuple total dominating set of maximum cardinality on several
classes of graphs, as well as finding general bounds and
characterizations. Also, we find some results on the upper $k$-tuple
total domination number of the Cartesian and cross product graphs.
\\[0.2em]

\noindent
Keywords: $k$-tuple total domination number, upper $k$-tuple total domination number,
Cartesian and cross product graphs, hypergraph, (upper) $k$-transversal number.
\\[0.2em]

\noindent
MSC(2010): 05C69.
\end{abstract}

\pagestyle{myheadings}
\markboth{\centerline {\scriptsize  A. P. Kazemi}}     {\centerline {\scriptsize A. P. Kazemi,~~~~~~~~~~~~~~~~~~~~~~~~~~~~~~~~~~~~~~~~~~~~~~~~~~~~~~~~~~~~~~~~~~~~~~~~      Upper $k$-tuple total domination in graphs}}


\section {\bf Introduction}

All graphs considered here are finite, undirected and simple. For
standard graph theory terminology not given here we refer to
\cite{West}. Let $G=(V,E) $ be a graph with the \emph{vertex set}
$V$ of \emph{order} $n(G)$ and the \emph{edge set} $E$ of
\emph{size} $m(G)$. The \emph{open neighborhood} of a vertex $v\in
V$ is $N_{G}(v)=\{u\in V\ |\ uv\in E\}$, while its cardinality is
the \emph{degree} of $v$ and denoted by $deg_G(v)$. The \emph{closed
neighborhood} of a vertex $v\in V$ is also $N[v]=N_{G}(v)\cup\{v\}$.
The \emph{minimum} and \emph{maximum degree} of $G$ are denoted by
$\delta =\delta (G)$ and $\Delta =\Delta (G)$, respectively. We
write $K_n$, $C_{n}$ and $P_{n}$ for a \emph{complete graph}, a
\emph{cycle}, and a \emph{path} of order $n$, respectively, while
$K_{n_1,...,n_p}$ denotes a \emph{complete $p$-partite graph}. Also for a subset $S\subseteq V$, $G[S]$ denotes the \emph{induced subgraph} of $G$ by $S$ in which $V(G[S])=S$ and for any two vertices $x,y\in S$, $xy\in E(G[S])$ if and only if $xy\in E(G)$.
\vspace{0.2 cm}


\begin{defn}
\label{kOPN} \emph{Let $k\geq 1$ be an integer and let $v\in
S\subseteq V$. A vertex $v'$ is called a }$k$-open private neighbor
\emph{of $v$ with respect to $S$, or simply a ($S,k$)}-opn \emph{of $v$ if
$v\in N_{G}(v')$ and $| N_{G}(v')\cap S| =k$, in other words, there
exists a $k$-subset $S_{v}\subseteq S$ containing $v$ such that
$N_{G}(v')\cap S=S_{v}$. The set
\[
opn_{k}(v;S)=\{v'\in V| \mbox{$v'$ is a ($S,k$)-opn of $v$}\}
\]
is called the } $k$-open private neighborhood set \emph{of $v$ with
respect to $S$. Also, a $k$-open private neighbor of $v$ with respect
to $S$ is called} external \emph{or} inner \emph{if the vertex is in $V-S$ or $S$, respectively.} 
\end{defn}


\textbf{Hypergraphs.} Hypergraphs are systems of sets which are conceived as natural
extensions of graphs. A \emph{hypergraph} $H=(V,E)$ is a set $V$ of
elements, called \emph{vertices}, together with a multiset $E$ of
arbitrary subsets of $V$, called \emph{edges}. For integer $k\geq
1$, a $k$-\emph{uniform hypergraph} is a hypergraph in which every
edge has size~$k$. Every simple graph is a $2$-uniform hypergraph.
For a graph $G=(V,E)$, $H_{G}=(V,C)$ denotes the \emph{open
neighborhood hypergraph} of $G$ with the vertex set $V$ and edge set
$C$ consisting of the open neighborhoods of vertices of $V$ in $G$.
\vspace{0.2 cm}

A \emph{transversal} in a hypergraph $H=(V,E)$ is a subset
$S\subseteq V$ such that $|S\cap e|\geq 1$ for every edge $e\in E$;
that is, the set $S$ meets every edge in $H$. The \emph{transversal
number} $\tau (H)$ of $H$ is the minimum size of a transversal in
$H$. In a natural way, Wanless et al. generalized the concept of transversal in a Latin square to $k$-transversal \cite{WCS}.

\begin{defn}
\emph{\cite{WCS} For any positive integer $k$, a }$k$-transversal \emph{or a} $k$-plex \emph{in a Latin square of order $n$ is a set of $nk$ cells, $k$ from each row, $k$ from each column, in which every symbol occurs exactly $k$ times. The maximum number of disjoint $k$-transversals in a Latin square $L$ is called its} $k$-transversal number \emph{and denoted by
$\tau_{k}(L)$. Obviously $\tau_{k}(L)\leq n/k$. A Latin squre $L$ has} a decomposition into disjoint $k$-transversals \emph{means $\tau_{k}(L)=n/k$}.
\end{defn}

In a similar way, we generalize the concept of transversal in a hypergraph to $k$-transversal as following: 

\begin{defn}
\emph{For any integer $k\geq 1$, a} $k$-transversal
\emph{in a hypergraph $H=(V,E)$ is a subset $S\subseteq V$ such that
$|S\cap e|\geq k$ for every edge $e\in E$; that is, every edge in
$H$ contains at least~$k$ vertices from the set $S$. The} $k$-transversal number $\tau _{k}(H)$ \emph{of $H$ is the minimum cardinality of a minimal $k$-transversal in $H$, while
the} upper $k$-transversal number $\Upsilon _{k}(H)$ \emph{of $H$ is defined as the
maximum cardinality of a minimal $k$-transversal in $H$}.
\end{defn}


\textbf{Domination.} Domination in graphs is now well studied in graph theory and the
literature on this subject has been surveyed and detailed in the two
books by Haynes, Hedetniemi, and Slater~\cite{hhs1, hhs2}. A set
$S\subseteq V$ is a \emph{dominating set} (resp. \emph{total
dominating set}) of $G$ if each vertex in $V\setminus S$ (resp. $V$)
is adjacent to at least one vertex of $S$. The \emph{domination
number} $\gamma (G)$ (resp. \emph{total domination number} $\gamma
_{t}(G)$) of $G$ is the minimum cardinality of a dominating set
(resp. total dominating set) of $G$. An extension of
total domination number introduced by Henning and Kazemi in \cite{HeKa09} (the reader can study \cite{HK,KazInflated,KazTOC,KazBeh,KPS} for more information).

\begin{defn}\label{kTDN}
\emph{\cite{HeKa09} Let $k\geq 1$ be an integer and let $G$ be a graph
with $\delta(G)\geq k$. A subset $S\subseteq V$ is called a} $k$-tuple
total dominating set, \emph{briefly $k$TDS, of $G$ if for each $x\in
V$, $| N(x)\cap S| \geq k$. The minimum number of vertices of a $k$TDS
of $G$ is called the }$k$-tuple total domination number \emph{of $G$
and denoted by $\gamma _{\times k,t}(G)$. A $k$TDS with cardinality
$\gamma _{\times k,t}(G)$ is called a} min-TDS \emph{of $G$}.
\end{defn}

Finding the maximum cardinality of the set of minimal subsets of the vertices (or edges or both) of a graph with a property is one of the important problems in graph theory. According to this fact, in this paper, we initiate the study of the problem of finding a minimal $k$-tuple total dominating set of maximum cardinality in a graph. This leads to our next definition.

\begin{defn}
\emph{The} upper $k$-tuple total domination number $\Gamma _{\times
k,t}(G)$ \emph{of $G$ is the maximum cardinality of a minimal $k$TDS
of $G$, and a minimal $k$TDS with cardinality
$\Gamma _{\times k,t}(G)$ is a} $\Gamma _{\times k,t}(G)$-set,
\emph{ or a} $\Gamma _{\times k,t}$-set of $G$. \emph{Also, we say that a graph $G$ is a} $\Gamma _{\times
k,t}$-external graph \emph{if it has a $\Gamma _{\times k,t}$-set
$S$ such that every vertex in it has an external $k$-open private
neighbor with respect to $S$}.
\end{defn}

Obviously, for every graph $G$ and every positive integer $k$, $%
\gamma _{\times k,t}(G)\leq \Gamma _{\times k,t}(G)$, and this bound
is sharp by $\gamma _{\times
k,t}(K_{n})=\Gamma _{\times k,t}(K_{n})=k+1$ when $1\leq k<n$. We remark that the
upper $1$-tuple total domination number $\Gamma _{\times 1,t}(G)$ is
the well-studied \emph{upper total domination number} $\Gamma
_{t}(G)$, while the upper $2$-tuple total domination number is known
as the \emph{upper double total domination number}. The redundancy
involved in upper $k$-tuple total domination makes it useful in many
applications.
\vspace{0.2 cm}

In this paper, as we said before, we initiate the study of the problem of finding a minimal
$k$-tuple total dominating set of maximum cardinality on several
classes of graphs, as well as finding general bounds and
characterizations. Also we present a Vizing-like conjecture on the upper
$k$-tuple total domination number, and prove it for a family of graphs. Proving 
\[
\Gamma
_{\times k\ell ,t}(G\times H)\geq \Gamma _{\times k,t}(G) \cdot  \Gamma
_{\times \ell ,t}(H)~ \mbox{  (for any } k,\ell\geq 1)
\]
is our next work in which $G\times H$ denotes the cross product of two graphs $G$ and $H$. Then we characterize graphs $G$ satisfying $\Gamma _{\times
k,t}(G)=\gamma _{\times k,t}(G)$, and show that for any graph $G$
with minimum degree at least $k$,\\
\verb"1." $\Gamma _{\times k,t}(G)=\Upsilon _{k}(H_{G})$, and \\
\verb"2." $\Gamma _{\times k,t}(G)=\gamma _{\times k,t}(G)$ if and only if $\Upsilon
_{k}(H_{G})=\tau _{k}(H_{G})$.
\vspace{0.2 cm}


We begin our discussion with the following useful observation. 
\begin{obs}
\label{obs} Let $k\geq 1$ be an integer and let $G$ be a graph of order $n$ with $\delta
(G)\geq k$. Then

\verb"i." $\gamma _{\times k,t}(G) \leq \Gamma _{\times k,t}(G)\leq n$,

\verb"ii." every $k$TDS $S$ of $G$ is minimal if and only if $opn_{k}(v;S)\neq \emptyset$ for every
vertex $v\in S$,

\verb"iii." all neighbors of every vertex of degree $k$ in $G$ belong to
every $k$TDS of $G,$ and

\verb"iv." if $H$ is a spanning subgraph of $G$ which has a
$\Gamma_{\times k,t}$-set that is also a minimal $k$TDS of $G$, then
$\Gamma _{\times k,t}(H)\leq \Gamma _{\times k,t}(G)$.
\end{obs}

Observation \ref{obs} (iii) implies the next proposition.


\begin{prop}\label{remark2}
For any $k$-regular graph $G$, $\Gamma _{\times k,t}(G)=n$.
\end{prop}

The converse of Proposition \ref{remark2} does not hold. For
example, if $G$ is the graph obtained by the union of two disjoint
complete graphs of order $k+1\geq 3$, with an edge between them,
then $G$ is not regular but $\Gamma_{\times k,t}(G)=2k+2$. The next two propositions are
useful for our investigations. 

\vspace{0.2 cm}
We recall that for any positive integer $k$, the $k$-\emph{join} $G\circ _{k}H$ of a graph $G$ to a graph $H$ with $\delta(H)\geq k$ is the graph obtained from the disjoint union of $G$ and $H$ by joining each
vertex of $G$ to at least $k$ vertices of $H$.

\begin{prop}
\label{Gamma t Pn } \emph{\cite{DHM}} For any path $P_{n}$ of order $n\geq 2$,
\[
\Gamma _{t}(P_{n})=2\lfloor(n+1)/3\rfloor.
\]
\end{prop}

\begin{prop}
\label{gamma xkt G=k+1 iff ...}
\emph{\cite{HeKa09}} Let $G$ be a graph with $\delta (G)\geq k$. Then $\gamma
_{\times k,t}(G)=k+1$ if and only if $G=K_{k+1}$ or $G=F\circ _{k}K_{k+1}$
for some graph $F$.
\end{prop}

\section{Cycles and complete mutipartite graphs}
\vskip0.2 true cm

In this section, we calculate the upper $k$-tuple total domination
number of a cycle and a complete multipartite graph.
Proposition \ref{remark2} implies $\Gamma _{\times 2,t}(C_{n})=n$.
The next proposition calculates $\Gamma _{t}(C_{n})$.

\begin{prop}
\label{Gamma t Cn} For any cycle $C_{n}$ of order $n\geq 3$,
\[
\Gamma _{t}(C_{n})=\left\{
\begin{array}{ll}
2\lfloor \frac{n}{3}\rfloor +1 & \mbox{if }n\equiv2\pmod{3},
\\
2\lfloor \frac{n}{3}\rfloor & \mbox{otherwise.}
\end{array}
\right.
\]
\end{prop}

\begin{proof}
Let $V(C_{n})=\{1,2,...,n\}$, and let $ij\in E(C_{n})$ if and only
if $j\equiv i+1\pmod{n}$. Let $S$ be a $\Gamma _{t}(C_{n})$-set. If
at least one vertex of any two consecutive vertices belongs to $S$,
then $n\equiv 0\pmod{3}$. Since, otherwise, $S$ will contain at
least three consecutive vertices of $V(C_n)$, which contracts the
minimality of $S$. Hence $|S|=2\lfloor\frac{n}{3}\rfloor$, when
$n\equiv 0\pmod{3}$. Now, assume there exist two consecutive
vertices, say $1$ and $n$, out of $S$. Then $S$ is
also a minimal TDS in the path $P_n=C_n-\{e\}$ in which $e=1n\in E(C_n)$. This
implies 
\[
\begin{array}{llll}
|S| & = & \Gamma _{t}(C_{n})\\
      &\leq & \Gamma_{t}(P_{n}) & \\
     & = & 2\lfloor(n+1)/3\rfloor & (\mbox{by Proposition \ref{Gamma t Pn }}) .
\end{array}
\]
Now since $\{3i+1,3i+2|
0\leq i\leq \lfloor \frac{n}{3}\rfloor -1\}$ is a minimal TDS in
$C_n$ with cardinality $\Gamma _{t}(P_{n})$ when
$n\not\equiv2\pmod{3}$, we obtain $\Gamma
_{t}(C_{n})=2\lfloor\frac{n+1}{3}\rfloor=2\lfloor \frac{n}{3}\rfloor
$. Now let $n\equiv2\pmod{3}$ and let $S$ be a minimal TDS of $C_{n}$
with cardinality $\Gamma _{t}(P_{n})$. Then there exist seven
consecutive vertices, say 1, 2, ..., 7, such that $S\cap
\{1,2,...,7\}=\{i\}$ which $i=2$ or 4. Since $S-\{i+1\}$ is a TDS
of $C_n$, we
obtain $|S|<\Gamma _{t}(P_{n})$, and so $\Gamma
_{t}(C_{n})\leq \Gamma _{t}(P_{n})-1$. Now since $\{3i+1,3i+2| 0\leq i\leq
\lfloor \frac{n}{3}\rfloor -1\}\cup \{n\}$ is a minimal TDS of $C_n$
with cardinality $\Gamma _{t}(P_{n})-1=2\lfloor\frac{n}{3}\rfloor
+1$, we obtain $\Gamma _{t}(C_{n})=2\lfloor\frac{n}{3}\rfloor +1$.
\end{proof}

\begin{thm}
\label{Gamma xkt complete multipartite} Let
$G=K_{n_{1},n_{2},...,n_{p}}$ be a complete $p$-partite graph with
$\delta(G)\geq k \geq 1$ in which $n_1\leq n_{2}\leq ...\leq n_{p}$. Then
\[
\Gamma _{\times k,t}(G)=k+\max \{x~|~ (\ell-1)x=k\mbox{ and }x\leq
\min \{k,n_{p-\ell+1},...,n_{p}\}\}.
\]
\end{thm}

\begin{proof}
Let $S$ be a minimal $k$TDS of $G=K_{n_{1},n_{2},...,n_{p}}$ and let $V=X_{1}\cup X_{2}\cup ...\cup X_{p}$ be the partition of the vertex set of $G$ to the $p$ independent sets $X_{1}$, $X_{2}$, $\cdots$, $X_{p}$ in which $|X_i|=n_i$ for each $i$ and $n_1\leq n_{2}\leq ...\leq n_{p}$. Let $I=\{i_j ~|~
j=1,..,\ell\}$ be an index subset of $\{1,2,..,p\}$ for some $2\leq
\ell\leq p$ such that $S\cap X_{i}\neq \emptyset $ if and only if
$i\in I$. Also assume $| S\cap X_{i_j}| =x_{i_j}$\ for each
$i_j\in I$, and $x_{i_1}\leq x_{i_2}\leq ...\leq x_{i_{\ell}}$. The
minimality of $S$ implies $x_{i_j}\leq k$ for each $i_j\in I$,
and there exists a $(\ell-1)$-subset $L\subseteq I$ such that
$\sum_{i_j \in L}x_{i_j}=k$. Then, by the minimality of $S$, $\sum_{i_j\in L}x_{i_j}=k$ for every $(\ell-1)$-subset $L\subseteq I$, and so $x_{i_1}=x_{i_2}=...=x_{i_{\ell}}$. Let
$x:=x_{i_1}=x_{i_2}=...=x_{i_{\ell}}$. Then $%
x_{i_1}+x_{i_2}+...+x_{i_{\ell}}=\ell x=k+x\leq \Gamma _{\times
k,t}(G)$ where $x\leq \min \{k,n_{i_1},...,n_{i_{\ell}}\}$, and so
\[
\begin{array}{lll}
\Gamma _{\times k,t}(G) & = & k+\max \{x~|~ (\ell-1)x=k\mbox{ and }x\leq \min \{k,n_{i_1},...,n_{i_{\ell}}\}\} \\
& = & k+\max \{x~|~ (\ell-1)x=k\mbox{ and }x\leq \min
\{k,n_{p-\ell+1},...,n_{p}\}\}.
\end{array}
\]
\end{proof}

\begin{cor}
\label{Gamma xkt complete multipartite=2k} Let $G=K_{n_{1},n_{2},...,n_{p}}$ be a complete $p$-partite graph. For any integer $k\geq 1$ if $|\{~i~|~n_{i}\geq k\}|\geq 2$, then $\Gamma _{\times k,t}(G)=2k.$
\end{cor}

In a similar way, the next theorem can be proved.

\begin{thm}
\label{Gamma xkt complete multipartite=<k+min} Let $G=K_{n_{1},n_{2},...,n_{p}}$ be a complete $p$-partite graph with
$\delta(G) \geq k \geq 1$ in which $n_1\leq n_{2}\leq ...\leq n_{p}$. Then
\[
\gamma _{\times k,t}(G)\leq k+\min \{x| (\ell-1)x=k\mbox{ and }x\leq
\min \{k,n_{1},...,n_{\ell}\}\}.
\]
\end{thm}

\section{Two upper bounds}
\vskip0.2 true cm

In this section, we present two upper bounds for the upper $k$-tuple
total domination number of a graph. The first is in terms of $k$,
the order and the minimum degree of the graph, and the second is in
terms of the upper $\ell $-tuple total domination number of the
graph for some $\ell <k$.

\begin{thm}
\label{Gamma xkt=<n-1 or =<n-delta+k} If $G$ is a graph of order $n$
with $\delta\geq k+1\geq 2$, then $\Gamma _{\times k,t}(G)\leq
n-\delta +k$, and this bound is sharp.
\end{thm}

\begin{proof}
Let $G$ be a graph of order $n$ with $\delta\geq k+1\geq 2$ and let $S$ be a $\Gamma _{\times k,t}(G)$-set. Then for every $v\in S$ there exist a $k$-subset $S_{v}\subseteq S$ and a vertex $v'\in V(G)$ such that $N_{G}(v')\cap S=S_{v}$, by Observation \ref{obs} (ii). If $v'\in S $, then
$N_{G}(v')-S_{v}\subseteq V(G)-S$, and so 
\[
\delta -k\leq deg(v')-k\leq n-| S| =n-\Gamma _{\times k,t}(G),
\]
which implies $\Gamma _{\times k,t}(G)\leq n-\delta +k$. If $v'\not\in S$, then
$v'$ is not adjacent to at least $| S| -k$ vertices of $S-S_{v}$, and so
\[
\delta \leq deg(v')\leq n-| S| +k-1=n-\Gamma _{\times k,t}(G)+k-1,
\]
which implies $\Gamma _{\times k,t}(G)<n-\delta +k$.
\vspace{0.2 cm}

The sharpness of this bound can be seen as following: Let $\delta \geq
k+1\geq 2$ be integers. Consider $b$ vertex-disjoint complete graphs
$K_{k+1}$ where $b\geq \lceil \frac{\delta }{k+1}\rceil $, and let
$H_{b}=K_{k+1}+...+K_{k+1}$ be the union of $b$ vertex-disjoint
complete graphs $K_{k+1}$. Also consider an empty graph $T$ with
$\delta -k$ vertices. Let $G_{b}=H_{b}\vee T$ be the \emph{join} of $H_b$ and $T$, which is the union of
$H_{b}$ and $T$ such that every vertex of $H_{b}$ is adjacent to all
vertices in $T$. Then $G_{b}$ is a connected graph of order
$n=b(k+1)+\delta -k$ with minimum degree $\delta $. Since $V(H_{b})$
is a minimal $k$TDS of $G_{b}$, we obtain
$\Gamma _{\times k,t}(G_b)\geq n-\delta +k$, and consequently $\Gamma
_{\times k,t}(G_b)=n-\delta +k$.
\end{proof}


\begin{thm}
\label{Gamma xkt=<Gamma x(k-ell)t+ell} Let $G$ be a graph with
$\delta \geq k\geq 1$. Let $L=\cap _{v\in S}S_{v}$ be a set of cardinality $\ell$ in which $S$ is a
$\Gamma _{\times k,t}(G)$-set and $S_v$ is the set given in
Definition \ref{kOPN}. If $\ell <k$, then
\[
\Gamma _{\times k,t}(G)\leq \Gamma _{\times (k-\ell ),t}(G)+\ell.
\]
\end{thm}

\begin{proof}
Let $S$ be a $\Gamma _{\times k,t}(G)$-set and let $L=\cap _{v\in
S}S_{v}$ be a set of cardinality $\ell$ in which $S_v$ is the set given in
Definition \ref{kOPN} and $\ell<k$. Since $S-L$ is a minimal
($k-\ell $)TDS of $G$, we obtain
\[
\begin{array}{lll}
\Gamma _{\times (k-\ell ),t}(G) & \geq & | S-L| \\
& = & | S| -\ell \\
& = & \Gamma _{\times k,t}(G)-\ell.
\end{array}
\]
\end{proof}

\section{The Cartesian product and a Vizing-like conjecture}
\vskip0.2 true cm
The \emph{Cartesian product} $G\, \Box \, H$ of two graphs $G$ and
$H$ is a graph with the vertex set $V(G)\times V(H)$ and two
vertices $(g_1,h_1)$ and $(g_2,h_2)$ are adjacent if and only if
either $g_1=g_2$ and $(h_1,h_2)\in E(H)$, or $h_1=h_2$ and
$(g_1,g_2)\in E(G)$. For more information on product graphs see \cite{IK}. The Cartesian product $K_n \Box K_m$ is known as the $n \times m$ \emph{rook's graph}, as edges represent possible moves by a rook on an $n \times m$ chess board. For example see Figure \ref{fi:rook34}.
\begin{figure}[htp]
\centering
\begin{tikzpicture}
\matrix[nodes={draw, thick, fill=black!20, circle},row sep=0.7cm,column sep=0.7cm] {
  \node(11){}; &
  \node(12){}; &
  \node(13){}; &
  \node(14){}; \\
  \node(21){}; &
  \node(22){}; &
  \node(23){}; &
  \node(24){}; \\
  \node(31){}; &
  \node(32){}; &
  \node(33){}; &
  \node(34){}; \\
};
\draw (11) to (12) to (13) to (14); \draw (11) to[bend left] (13); \draw (12) to[bend left] (14); \draw (11) to[bend left] (14);
\draw (21) to (22) to (23) to (24); \draw (21) to[bend left] (23); \draw (22) to[bend left] (24); \draw (21) to[bend left] (24);
\draw (31) to (32) to (33) to (34); \draw (31) to[bend left] (33); \draw (32) to[bend left] (34); \draw (31) to[bend left] (34);
\draw (11) to (21) to (31) to[bend left] (11);  \draw (12) to (22) to (32) to[bend left] (12);  \draw (13) to (23) to (33) to[bend left] (13);  \draw (14) to (24) to (34) to[bend left] (14);
\end{tikzpicture}
\caption{The $3 \times 4$ rook's graph, i.e., $K_3 \Box K_4$.}\label{fi:rook34}
\end{figure}

Now for integers $n\geq m\geq k+1\geq 3$ we consider the $n\times m$ rook's graph $K_{n}\Box K_{m}$ with the vertex set $V=\{(i,j)~|~1\leq i \leq n,~1\leq j \leq m\}$. Since the set $\{(i,j)~|~1\leq i \leq n,~ 1\leq j \leq k\}$ is a minimal $k$TDS of $K_{n}\Box K_{m}$, we have the following proposition.

\begin{prop}\label{Gamma_{xk,t}(Kn Box Km >=kn}
For any integers $n\geq m\geq k+1\geq 3$, $\Gamma_{\times k,t}(K_n\Box K_m)\geq kn$.
\end{prop}
As we will see in Proposition \ref{Gamma_{xk,t}(K_{k+1} Box K_{k+1}} for $n=m=k+1\geq 3$, we guess equality holds in Proposition \ref{Gamma_{xk,t}(Kn Box Km >=kn}.
\begin{prop}\label{Gamma_{xk,t}(K_{k+1} Box K_{k+1}}
For any integer $k\geq 2$, $\Gamma_{\times k,t}(K_{k+1}\Box K_{k+1})=k(k+1)$.
\end{prop}

\begin{proof}
Let $V(K_{k+1}\Box K_{k+1})=\{(i,j)~|~1\leq i,j \leq k+1\}$ in which $k\geq 2$. We know $\Gamma_{\times k,t}(K_{k+1}\Box K_{k+1})\geq k(k+1)$ by Proposition \ref{Gamma_{xk,t}(Kn Box Km >=kn}. Now let 
\[
S=\bigcup_{1\leq i \leq k+1}S_i^r=\bigcup_{1\leq j \leq k+1}S_j^c
\]
be a minimal $k$TDS of $K_{k+1}\Box K_{k+1}$ with cardinality more than $k(k+1)$ in which 
\[
S_i^r=S\cap \{(i,j)~|~1\leq j \leq k+1\},
\]
\[
S_j^c=S\cap \{(i,j)~|~1\leq i \leq k+1\}.
\]
Then $|S_i^r|\geq k$ and $|S_j^c|\geq k$ for each $i$ and each $j$, and also $|S_t^r|= k+1$ and $|S_{\ell}^c|=k+1$ for some $t$ and some $\ell$. Now since $S-\{(t,\ell)\}$ is a $k$TDS of $K_{k+1}\Box K_{k+1}$ whic contradicts the minimality of $S$, we obtain $\Gamma_{\times k,t}(K_{k+1}\Box K_{k+1})=k(k+1)$. See Figure \ref{fi:rook44} for an example.
\end{proof}
\begin{figure}[htp]
\centering
\begin{tikzpicture}
\matrix[nodes={draw, thick, fill=black!20, circle},row sep=0.7cm,column sep=0.7cm] {
  \node[fill=black!80](11){}; &
  \node[fill=black!80](12){}; &
  \node[fill=black!80](13){}; &
  \node(14){}; \\
  \node[fill=black!80](21){}; &
  \node[fill=black!80](22){}; &
  \node[fill=black!80](23){}; &
  \node(24){}; \\
  \node[fill=black!80](31){}; &
  \node[fill=black!80](32){}; &
  \node[fill=black!80](33){}; &
  \node(34){}; \\
  \node[fill=black!80](41){}; &
  \node[fill=black!80](42){}; &
  \node[fill=black!80](43){}; &
  \node(44){}; \\
};
\draw (11) to (12) to (13) to (14); \draw (11) to[bend left] (13); \draw (12) to[bend left] (14); \draw (11) to[bend left] (14);
\draw (21) to (22) to (23) to (24); \draw (21) to[bend left] (23); \draw (22) to[bend left] (24); \draw (21) to[bend left] (24);
\draw (31) to (32) to (33) to (34); \draw (31) to[bend left] (33); \draw (32) to[bend left] (34); \draw (31) to[bend left] (34);
\draw (11) to (21) to (31) to[bend left] (11);  \draw (12) to (22) to (32) to[bend left] (12);  \draw (13) to (23) to (33) to[bend left] (13);  \draw (14) to (24) to (34) to[bend left] (14);
\draw (41) to (42) to (43) to (44); \draw (41) to[bend left] (43); \draw (42) to[bend left] (44); \draw (41) to[bend left] (44);
\draw (31) to (41) to [bend left] (21); \draw (41) to[bend left] (11); 
\draw (32) to (42) to [bend left] (22); \draw (42) to[bend left] (12); 
\draw (33) to (43) to [bend left] (23); \draw (43) to[bend left] (13);
\draw (34) to (44) to [bend left] (24); \draw (44) to[bend left] (14); 

\end{tikzpicture}
\caption{The dark vertices highlight a minimal $3$TDS of $K_4 \Box K_4$ with maximum cardinality.}\label{fi:rook44}
\end{figure}


In 1963, more formally in 1968, Vizing \cite{Viz} made an elegant conjecture that has subsequently become one the most famous open problems in domination theory.

\begin{conj}[Vizing's Conjecture]
For any graphs $G$ and $H$, 
\[
\gamma(G)\cdot \gamma(H) \leq \gamma(G\Box H).
\]
\end{conj}

Over more than fifty years (see \cite{BDG} and references therein), Vizing's Conjecture has been shown to hold for certain restricted classes of graphs, and furthermore, upper and lower bounds on the inequality have gradually tightened.  Additionally, research has explored inequalities (including Vizing-like inequalities) for different forms of domination \cite{hhs2}. A significant breakthrough occurred in 2000, when Clark and Suen \cite{CS1} proved that 
\[
\gamma (G)\cdot  \gamma (H)\leq 2\gamma (G\Box H)
\]
which led to the discovery of a Vizing-like inequality for total domination \cite{HR,HPT}, i.e.,
\begin{equation}\label{eq:HRbound}
\gamma_{t}(G)\cdot \gamma_{t}(H) \leq 2 \gamma_{t}(G \Box H),
\end{equation}
as well as for paired \cite{BHRo,CM,HJ}, and fractional domination \cite{FRDM}, and the $\{k\}$-domination function (integer domination) \cite{BM,CMH,HL}, and total $\{k\}$-domination function \cite{HL}. In 1996, Nowakowski and Rall
in \cite{NR} made the following Vizing-like conjecture for the upper
domination of Cartesian products of graphs.

\begin{conj}[Nowakowski-Rall's Conjecture] \label{Nowakowski-Rall's Conjecture}
For any graphs $G$ and $H$, 
\[
\Gamma (G)\cdot \Gamma (H)\leq \Gamma (G\, \Box \, H).
\]
\end{conj}

A beautiful proof of the Nowakowski-Rall's Conjecture was recently
found by Bre\v{s}ar \cite{Br}. Also Paul Dorbec et al. in \cite{DHR}\ proved that for any graphs $G$ and
$H$ with no isolated vertices, 
\begin{equation}\label{eq:DHRbound}
\Gamma _{t}(G)\cdot \Gamma _{t}(H)\leq 2\Gamma _{t}(G\, \Box \, H),
\end{equation}

We guess (\ref{eq:DHRbound}) can be extended as follows: 

\begin{conj} 
\label{Adel's Conj.} \emph{(\textbf{Vizing-like conjecture for upper $k$-tuple total domination})}\\
For any integer $k \geq 2$ and any graphs $G$ and $H$ with minimum degrees at least $k$,
\[
\Gamma_{\times k,t}(G) \cdot \Gamma_{\times k,t}(H) \leq \frac{k+1}{k} \cdot \Gamma_{\times k,t} (G \, \Box \,  H).
\]
\end{conj}
Let $G_1$, $G_2$, $\cdots$, $G_n$  and $H_1$, $H_2$, $\cdots$, $H_m$ be respectively the all connected components of two graphs $G$ and $H$ which have minimum degrees at least $k\geq 2$. Then $G\Box H$ is a disconnected graph with the connected components $G_i\Box H_j$ for $1\leq i \leq n$ and $1\leq j\leq m$. By the truth of Conjecture \ref{Adel's Conj.} for connected graphs, since
\[
\begin{array}{lll}
\Gamma_{\times k,t}(G\Box H) & = & \sum_{1\leq i \leq n}\sum_{1\leq j \leq m} \Gamma_{\times k,t}(G_i\Box H_j) \\
& \geq & \sum_{1\leq i \leq n}\sum_{1\leq j \leq m} \frac{k}{k+1}\cdot \Gamma_{\times k,t}(G_i)\cdot \Gamma_{\times k,t}(H_j) \\
& = &  \frac{k}{k+1}\cdot (\sum_{1\leq i \leq n}\Gamma_{\times k,t}(G_i))\cdot (\sum_{1\leq j \leq m}\Gamma_{\times k,t}(H_j)) \\
& = &  \frac{k}{k+1}\cdot \Gamma_{\times k,t}(G))\cdot \Gamma_{\times k,t}(H) 
\end{array}
\]
we may conclude that Conjecture \ref{Adel's Conj.} is true for disconnected graphs. Proposition \ref{Gamma_{xk,t}(K_{k+1} Box K_{k+1}} shows the bound in Conjecture \ref{Adel's Conj.}, if true, is best possible. Theorem \ref{Conj.for.external graph}, which is obtained by Theorem \ref{Gamma xkt GH>=Gamma xkt G.Gamma xkt H}, shows that Conjecture \ref{Adel's Conj.} is true for a family of graphs.


\begin{thm} \label{Gamma xkt GH>=Gamma xkt G.Gamma xkt H}
For any two $\Gamma_{\times k,t}$-external graphs $G$ and $H$ with minimum degree at least $k\geq 2$,
\[
\Gamma _{\times k,t}(G\, \Box \, H)\geq \max \{\Gamma _{\times
k,t}(G)\cdot |V(H)| ,\Gamma _{\times k,t}(H)\cdot|V(G)| \}.
\]
\end{thm}

\begin{proof} 
Let $G$ and $H$ be two  $\Gamma_{\times k,t}$-external graphs with minimum degree at least $k\geq 2$, and let $\Gamma _{\times k,t}(G)\cdot |V(H)|=\max \{\Gamma _{\times k,t}(G)\cdot |V(H)| ,\Gamma _{\times
k,t}(H)\cdot|V(G)| \}$. Assume $D_{G}$ is a $\Gamma _{\times k,t}(G)$-set in which every
vertex of it has an external ($D_{G},k$)-opn. Obviousely, $D=D_{G}\times V(H)$ is a $k$TDS of $G\, \Box \, H$. To show that $D$ is minimal, it is sufficient to prove 
\[
opn_{k}((v,w);D)\neq \emptyset  \mbox{ for any vertex } (v,w)\in D.
\]
Let
$v^{\prime }\in opn_{k}(v;D_{G})\cap (V(G)-D_{G})$. Then
$N_{G}(v^{\prime })\cap
D_{G}=\{v,v_{1},v_{2},...,v_{k-1}\}$ for some vertices $v_{1},v_{2},...,v_{k-1}\in D_{G}$, and so
\[
\begin{array}{lll}
N_{G\, \Box \, H}((v^{\prime },w))\cap D & = & ((N_{G}(v^{\prime
})\times
\{w\})\cup (\{v^{\prime }\}\times N_{H}(w))\cap D \\
& = & (N_{G}(v^{\prime })\cap D_{G})\times \{w\}\cup (\emptyset \times
N_{H}(w) \\
& = & \{(v,w),(v_{1},w),(v_{2},w),...,(v_{k-1},w)\},
\end{array}
\]
which implies $opn_{k}((v,w);D)\neq \emptyset $ for every vertex $(v,w)\in D$. Hence
\[
\begin{array}{lll}
\Gamma _{\times k,t}(G\, \Box \, H) & \geq & | D| \\
& \geq & \Gamma _{\times k,t}(G)\cdot | V(H)|\\
&  =   & \max \{\Gamma _{\times k,t}(G)\cdot |V(H)| ,\Gamma _{\times k,t}(H)\cdot|V(G)| \}.
\end{array}
\]
\end{proof}

\begin{thm}\label{Conj.for.external graph}
Let $G$ be a $\Gamma_{\times k,t}$-external graph with $\delta (G)
\geq k\geq 2$. Then for any graph $H$ with $\delta (H) \geq k$,
\[
\Gamma _{\times k,t}(G\, \Box \, H)\geq \Gamma _{\times
k,t}(G)\cdot\Gamma _{\times k,t}(H).
\]
\end{thm}


The proof of Theorem \ref {Gamma xkt GH>=Gamma xkt G.Gamma xkt H} with Proposition \ref{remark2} and Theorem
\ref{Gamma xkt=<n-1 or =<n-delta+k} imply next theorem.

\begin{thm}  \label{Gamma xkt GH>=max{Gamma xkt G,(Gamma xkt H +1)}}
Let $G$ be a $\Gamma_{\times k,t}$-external graph, and let $H$ be an
arbitrary graph. Then the following statements hold.

\verb"i." If $\delta (H)\geq k+1$, then $\Gamma _{\times k,t}(G\,
\Box \, H)\geq \Gamma _{\times k,t}(G)(\Gamma _{\times
k,t}(H)+\delta (H)-k)$.

\verb"ii." If $H$ is $k$-regular, then $\Gamma _{\times k,t}(G\,
\Box \, H)\geq \Gamma _{\times k,t}(G) \cdot \Gamma _{\times k,t}(H)$.

\verb"iii." If $H$ is not $k$-regular and $\delta (H)=k$, then
$\Gamma _{\times k,t}(G\, \Box \, H)\geq \Gamma _{\times
k,t}(G)(\Gamma _{\times k,t}(H)+1)$.
\end{thm}


\section{The cross product of graphs}
\vskip0.2 true cm

In this section, we study the upper $k$-tuple total domination
number of the cross product of two graphs. First we recall that the \emph{cross product} (also known as
the \emph{direct product}, \emph{tensor product},  \emph{categorical
product}, and \emph{conjunction} in the literature) $G \times H$ has
$V(G)\times V(H)$ as vertex set and two vertices $(g_1,h_1)$ and
$(g_2,h_2)$ are adjacent if and only if $(g_1,g_2)\in E(G)$ and
$(h_1,h_2)\in E(H)$. For example see Figure \ref{fi:crossK3K4}.
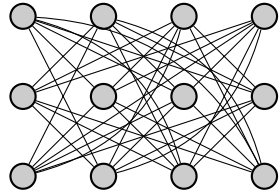
\begin{figure}[htp]
\centering
\begin{tikzpicture}
\matrix[nodes={draw, thick, fill=black!20, circle},row sep=0.7cm,column sep=0.7cm] {
  \node(11){}; &
  \node(12){}; &
  \node(13){}; &
  \node(14){}; \\
  \node(21){}; &
  \node(22){}; &
  \node(23){}; &
  \node(24){}; \\
  \node(31){}; &
  \node(32){}; &
  \node(33){}; &
  \node(34){}; \\
};
\draw (11) to (22) to (33) to (24) to (13) to (32) to (21) to (12) to (23) to (34) to (13) to (22) to (31) to (12) to (33) to (14) to (23) to (32) to (11);
\draw (11) to (23);
\draw (11) to (24);
\draw (11) to[bend left] (33);
\draw (11) to [bend left](34);
\draw (22) to (34);
\draw (12) to (24);
\draw (12) to [bend left](34);
\draw (13) to (21);
\draw (13) to [bend left](31);
\draw (14) to (21);
\draw (14) to (22);
\draw (14) to (31);
\draw (14) to [bend left](32);
\draw (21) to (33);
\draw (21) to (34);
\draw (23) to (31);
\draw (24) to (31);
\draw (24) to (32);
\end{tikzpicture}
\caption{The $K_3 \times K_4$.}\label{fi:crossK3K4}
\end{figure}

\begin{thm} \label{Gamma xkellt GxH>=Gamma xkt G,Gamma xellt H}
If $G$ and $H$ are graphs satisfying $\delta (G)\geq k\geq 1$ and $\delta
(H)\geq \ell \geq 1$, then
\[
\Gamma _{\times k\ell ,t}(G\times H)\geq \Gamma _{\times k,t}(G)\cdot\Gamma_{\times \ell ,t}(H).
\]
\end{thm}

\begin{proof}
Let $D_{G}$ and $D_{H}$ be two $\Gamma _{\times k,t}$-sets of $G$ and $H$, respectively. For a vertex $(u,v)\in V(G\times H)$, let $D_{G,u}=D_{G}\cap N_{G}(u)$ and $D_{H,v}=D_{H}\cap N_{H}(v)$. Since $D_{G}$ is a $k$TDS of $G$ and $D_{H}$ is a $\ell $TDS of $H$, we have $|D_{G,u}|\geq k$ and
$|D_{H,v}|\geq \ell $, and so $|D_{G,u}\times D_{H,v}|\geq k\ell $. Now by knowing
\[
\begin{array}{lll}
D_{G,u}\times D_{H,v} & \subseteq & N_{G}(u)\times N_{H}(v)\\
                                    &=&  N_{G\times H}((u,v)),
\end{array}
\]
we conclude the Cartesian product $D_{G}\times D_{H}$ of $D_{G}$ and $D_{H}$ is a $k\ell $TDS of
$G\times H$. To prove the minimality of $D_{G}\times D_{H}$ let $(a,b)\in
D_{G}\times D_{H}$. Then $a\in D_{G}$\ and $b\in D_{H}$ and the
minimality of $D_{G}$ and $D_{H}$ imply
\[
N_{G}(a')\cap D_{G}=S_{a}~\mbox{ for some vertex }a'\in V(G) \mbox{ and some } k \mbox{-subset } S_{a}\subseteq D_{G},\mbox{ and}
\]
\[
N_{H}(b')\cap D_{H}=S_{b}~\mbox{ for some vertex }b'\in V(H) \mbox{ and some } \ell\mbox{-subset } S_{b}\subseteq D_{H},
\]
and so
\[
N_{G\times H}((a',b'))\cap (D_{G}\times D_{H})=S_{a}\times S_{b}
\]
for the vertex $(a',b')\in V(G\times H)$ and the $k\ell$-subset $S_{a}\times S_{b}$. Hence $D_{G}\times D_{H}$ is a minimal $k\ell $TDS of $G\times H$, and so
\[
\begin{array}{lll}
\Gamma _{\times k\ell ,t}(G\times H) & \geq & | D_{G}\times D_{H}| \\
& = & | D_{G}| \cdot | D_{H}| \\
& = & \Gamma _{\times k,t}(G)\cdot \Gamma _{\times \ell ,t}(H).

\end{array}
\]
\end{proof}

\begin{cor}
\label{Gamma xkellt GxH>=max{Gamma xkt G,Gamma t H,Gamma xellt H.Gamma t G}}
If $G$ and $H$ are graphs satisfying $\delta (G)\geq \delta (H)\geq k\geq 1$%
, then
\[
\Gamma _{\times k,t}(G\times H)\geq \max \{\Gamma _{\times k,t}(G)\cdot \Gamma
_{t}(H),\Gamma _{\times k,t}(H)\cdot \Gamma _{t}(G)\}.
\]
\end{cor}

Next proposition shows that the bound given in Theorem
\ref{Gamma xkellt GxH>=Gamma xkt G,Gamma xellt H} is tight.

\begin{prop} \label{Gamma xkt KnxK2}
For any integers  $1\leq k\leq n-1$, $\Gamma _{\times k,t}(K_{n}\times K_{2})=2k+2.$
\end{prop}

\begin{proof}
For integers $1\leq k\leq n-1$ let $K_{n}\times K_{2}$ be the cross product of $K_n$ and $K_2$ with $V(K_{n}\times K_{2})=V_{1}\cup V_{2}$ in which $V_{i}=\{1,2,...,n\}\times \{i\}$ for $i=1,2$. For a minimal $k$TDS $S$ of $K_{n}\times K_{2}$ with maximum cardinality, let $S_{i}=S\cap V_{i}$ for $i=1,2$, and $| S_{1}| \geq | S_{2}| $. Obviousely $|S_{i}| \geq k$ for each $i$, and the minimality of $S$ implies $| S_{2}| \leq k+1$. Furthermore since $S$ has maximum cardinality, $| S_{2}|
=k+1$. If $| S_{1}|
>k+1$, then for any vertex $v\in
S_{1}-S_{2}^{\prime }$ the set $S-\{v\}$ is a $k$TDS of $K_{n}\times
K_{2}$ in which $S_{2}^{\prime }=\{(a,1)| (a,2)\in S_{2}\}$, a
contradiction. Hence $| S_{1}| =| S_{2}| =k+1$, and so $%
\Gamma _{\times k,t}(K_{n}\times K_{2})\leq 2k+2$. Now equality can be obtained by Corollary \ref{Gamma xkellt GxH>=max{Gamma xkt G,Gamma t H,Gamma
xellt H.Gamma t G}}. Figure \ref{fi:Gamma_{times 2,t}K4K2} shows a minimal $2$TDS of $K_4 \times K_2$ with maximum cardinality.
\end{proof}
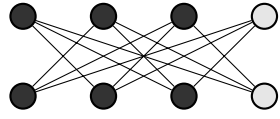
\begin{figure}[htp]
\centering
\begin{tikzpicture}
\matrix[nodes={draw, thick, fill=black!10, circle},row sep=0.7cm,column sep=0.7cm] {
\node[fill=black!80](11){}; &
\node[fill=black!80](12){}; &
\node[fill=black!80](13){}; &
\node(14){}; \\
\node[fill=black!80](21){}; &
\node[fill=black!80](22){}; &
\node[fill=black!80](23){}; &
\node(24){}; \\
};
\draw (11) to (22) to (13) to (24) to (12) to (21) to (14) to (23) to (11);
\draw (11) to (24);
\draw (12) to (23);
\draw (13) to (21);
\draw (14) to (22);
\end{tikzpicture}
\caption{The dark vertices highlight a minimal $2$TDS of $K_4 \times K_2$ with maximum cardinality.}\label{fi:Gamma_{times 2,t}K4K2}
\end{figure}
As a natural question we may ask the next question.

\begin{ques}
For any integers  $n,m\geq 2$ such that $\max\{n,m\}\geq k+1$, whether
\[
\Gamma _{\times k,t}(K_{n}\times K_{m})=2k+2?
\]
\end{ques}

Now we present a lower bound for the upper $k$-tuple total
domination number of the cross product of two complete multipartite
graphs.

\begin{prop} \label{Gamma xkt GxH>=4k}
Let $G\times H$ be the cross product of two complete multipartite graphs $G=K_{t_{1},t_{2},...,t_{m}}$ and $H=K_{s_{1},s_{2},...,s_{n}}$ with $\delta (G\times H)\geq k$. If
\[
\sum_{1\leq \ell\leq n}t_{i}s_{\ell}\geq \sum_{1\leq \ell\leq n}t_{j}s_{\ell}\geq 2k ~\mbox{ for some }~1\leq i\neq j\leq m,~ \mbox{ or}
\]
\[
\sum_{1\leq i\leq m}s_{\ell}t_{i}\geq \sum_{1\leq i\leq m}s_{r}t_{i}\geq 2k ~\mbox{ for some }~1\leq \ell \neq r\leq m,
\]
then $\Gamma _{\times k,t}(G\times H)\geq 4k$.
\end{prop}

\begin{proof}
Let $G=K_{t_{1},t_{2},\cdots ,t_{m}}$ be a complete $m$-partite graphs which has the partiotion $V(G)=X_{1}\cup X_{2}\cup ...\cup X_{m}$ to the disjoint independent sets $X_{1}$, $X_{2}$, $\cdots$, $X_{m}$ in which $| X_{i}| =t_{i}$ for each $i$. Similarly, let $H=K_{s_{1},s_{2},\cdots ,s_{n}}$ be a complete $n$-partite graphs which has the partiotion $V(H)=Y_{1}\cup Y_{2}\cup ...\cup Y_{n}$ to the disjoint independent sets $Y_{1}$, $Y_{2}$, $\cdots$, $Y_{n}$ in which $|Y_{i}| =s_{i}$ for each $i$. Then $V(G\times H)=\bigcup_{1\leq i\leq m,\mbox{ }1\leq j\leq n} (X_{i}\times Y_{j})$
is the partition of the vertex set of $G\times H$ to the independent sets $X_{i}\times Y_{j}$. Without loss of
generality, we may assume $m\geq n\geq 2$ and
\[
\sum_{1\leq \ell\leq n}t_{1}s_{\ell}\geq \sum_{1\leq \ell\leq
n}t_{2}s_{\ell}\geq 2k.
\]
For $1\leq i\leq r$, let $k_{i}\leq \min \{t_{1}s_{i},t_{2}s_{i},t_{1}s_{i+r},t_{2}s_{i+r}\}$ be a positive integer such that $k=k_{1}+\cdots+k_{r}$. Now we choose a subset $S$ of $V(G\times H)$ such that $| S\cap
(X_{1}\times Y_{i})| =k_{i}$ for each $i$. It can be easily seen that $S$ is a minimal $k$TDS of $G\times
H$, and so $\Gamma _{\times k,t}(G\times H)\geq 4k$.
\end{proof}

We think that the finding some complete multipartite graphs $G$ and $H$ with $\Gamma _{\times k,t}(G\times H)= 4k$ is a good problem to work.

\section{Upper $k$-transversal in hypergraphs}
\vskip0.2 true cm

In this section, we show that the problem of finding upper $k$-tuple
total dominating sets in graphs can be translated to the problem of
finding upper $k$-transversal in hypergraphs. We recall that $H_{G}$ denotes the
open neighborhood hypergraph of a graph $G$.

\begin{thm}\label{Gamma xkt G=Upsilon k H_G}
If $G$ is a graph with $\delta(G)\geq k\geq 1$, then $\Gamma _{\times
k,t}(G)=\Upsilon _{k}(H_{G})$.
\end{thm}

\begin{proof}
Since every $k$TDS of $G$ contains at least $k$ vertices
from the open neighborhood of each vertex in $G$, we conclude every $k$TDS of $G$ is a
$k$-transversal in $H_{G}$. On the other hand, we know every $k$-transversal
in $H_{G}$ contains at least $k$ vertices from the open neighborhood
of each vertex of $G$, and so is a $k$TDS of $G$. This shows that we have proved that a
vertex subset $S$ is a $k$TDS of $G$ if and only if it is a
$k$-transversal in $H_{G}$, and so $\Gamma _{\times k,t}(G)=\Upsilon _{k}(H_{G})$.
\end{proof}

The authors in \cite{HeKa09} proved the problem of finding $k$-tuple total
dominating sets in graphs can be translated to the problem of finding $k$%
-transversal in hypergraphs, that is, for every integer $k\geq 1$ and every
graph $G$ with minimum degree $k$, $\gamma _{\times k,t}(G)=\tau _{k}(H_{G})$.
This fact and Theorem \ref{Gamma xkt G=Upsilon k H_G} imply the next theorem.

\begin{thm} \label{Gamma xkt G=gamma xkt G iff Upsilon k H_G=tau k H_G}
For any graph $G$ with $\delta(G)\geq k\geq 1$, $$\Gamma _{\times
k,t}(G)=\gamma _{\times k,t}(G) \mbox{ if and only if } \Upsilon
_{k}(H_{G})=\tau _{k}(H_{G}).$$
\end{thm}
As we saw before, Proposition \ref{gamma xkt G=k+1 iff ...} characterize graphs $G$ satisfying
$\gamma _{\times k,t}(G)=k+1$. The next theorem characterizes graphs $G$ satisfying
$\gamma _{\times k,t}(G)=m$ for each $m\geq k+2\geq 3$. We note that in the next three theorems, $K'_m$ denotes a simple graph of order $m$ which has minimum degree at least $k$. 

\begin{thm} \label{gamma xkt G=m iff ...}
Let $G$ be a graph with $\delta (G)\geq k\geq 1$, and let $m\geq
k+2$ be an integer. Then
$\gamma _{\times k,t}(G)=m$ if and only if $G=K_{m}^{\prime }$ or $%
G=F\circ _{k}K_{m}^{\prime }$ in which $m$ is minimum in
\[
T=\{t~|~G=F'\circ _{k}K_{t}'\mbox{ for some graphs }F'\mbox{ and }K'_{t}\},
\]
and $F=G-K_{m}'$.
\end{thm}

\begin{proof} 
Let $G$ be a graph with $\delta (G)\geq k\geq 1$, and let $S$ be a min-$k$TDS of $G=(V,E)$ with
cardinality $m\geq k+2$. Then $G[S]=K_{m}'$ for some graph
$K_{m}'$ (because every vertex has at least $k$
neighbors in $S$). If $| V| =m$, then
$G=K_{m}'$. Otherwise, let $F=G[V-S]$. Since every vertex in $V-S$
has at least $k$ neighbors in $S$, we conclude $G=F\circ
_{k}K_{m}^{\prime }$, and by the definition of the $k$-tuple total
domination number, $m$ is minimum in $T$.
\vspace{0.2 cm}

Conversely, let $G=K_{m}'$ or $G=F\circ _{k}K_{m}'$, in which $m$ is
minimum in $T$, and let $F=G-K_{m}'$. Then $\gamma _{\times k,t}(G)\leq m$ because $V(K_{m}')$ is a $k$TDS with cardinality $m$. Now if $\gamma_{\times k,t}(G)=m'$ for some $m'<m$, then, by the previous discussion, $G=F'\circ _{k}K_{m'}'$ for some graph $F'$ and some graph $K_{m'}'$, which contradicts the minimality of $m$. This implies $\gamma_{\times k,t}(G)=m$.
\end{proof}

Proposition \ref {gamma xkt G=k+1 iff ...} and Theorem \ref {gamma xkt G=m iff ...} imply the next theorem. 


\begin{thm} \label{Gamma xkt G=gamma xkt G iff ...}
For any graph  $G$ with $\delta (G)\geq k\geq 1$, $\Gamma
_{\times k,t}(G)=\gamma _{\times k,t}(G)$ if and only if $G=K_{m}'$ or $G=F\circ _{k}K_{m}'$ in which $m$
is minimum in
\[
T=\{t~|~t\geq m+1,~ G=F'\circ _{k}K_{t}'\mbox{ for some graphs }F'\mbox{ and }K'_{t} \},
\]
and $F=G-K_{m}'$.
\end{thm}

Now by Theorems \ref {Gamma xkt G=gamma xkt G iff Upsilon k H_G=tau k H_G} and \ref {Gamma xkt G=gamma xkt G iff ...}, we conclude:
\begin{thm} \label{Upsilon k G=tau k G iff ...}
For any integer $k\geq 1$ and any hypergraph $H$, $\Upsilon
_{k}(H)=\tau _{k}(H) \mbox{ if and only if } H=H_G$, in which $G$ is $K_{m}'$ or $F\circ _{k}K_{m}'$ for some graph $K_{m}'$ and $m$ is minimum in
\[
T=\{t~|~t\geq m+1,~ G=F'\circ _{k}K_{t}'\mbox{ for some graphs }F'\mbox{ and }K'_{t} \},
\]
and $F=G-K_{m}'$.
\end{thm}


\end{document}